\documentclass[12pt]{amsart}
\usepackage{amsmath, amsthm, amssymb, amsfonts}
\usepackage{mathrsfs}
\usepackage{amscd, tikz-cd}
\usepackage{hyperref}
\usepackage{ulem}
\usepackage{xcolor}
\usepackage{braket}
\usepackage{fullpage}
\usepackage{graphicx}%
\usepackage{multirow}%
\usepackage{amsmath,amssymb,amsfonts}%
\usepackage{amsthm}
\usepackage{mathrsfs}%
\usepackage[title]{appendix}%
\usepackage{xcolor}%
\usepackage{textcomp}%
\usepackage{manyfoot}%
\usepackage{booktabs}%
\usepackage{algorithm}%
\usepackage{algorithmicx}%
\usepackage{algpseudocode}%
\usepackage{listings}%
\usepackage{mathrsfs}
\usepackage{hyperref}
\makeatletter
\newcommand*{\da@rightarrow}{\mathchar"0\hexnumber@\symAMSa 4B }
\newcommand*{\da@leftarrow}{\mathchar"0\hexnumber@\symAMSa 4C }
\newcommand*{\xdashrightarrow}[2][]{%
  \mathrel{%
    \mathpalette{\da@xarrow{#1}{#2}{}\da@rightarrow{\,}{}}{}%
  }%
}
\newcommand{\xdashleftarrow}[2][]{%
  \mathrel{%
    \mathpalette{\da@xarrow{#1}{#2}\da@leftarrow{}{}{\,}}{}%
  }%
}
\newcommand*{\da@xarrow}[7]{%
  \sbox0{$\ifx#7\scriptstyle\scriptscriptstyle\else\scriptstyle\fi#5#1#6\m@th$}%
  \sbox2{$\ifx#7\scriptstyle\scriptscriptstyle\else\scriptstyle\fi#5#2#6\m@th$}%
  \sbox4{$#7\dabar@\m@th$}%
  \dimen@=\wd0 %
  \ifdim\wd2 >\dimen@
    \dimen@=\wd2 %
  \fi
  \count@=2 %
  \def\da@bars{\dabar@\dabar@}%
  \@whiledim\count@\wd4<\dimen@\do{%
    \advance\count@\@ne
    \expandafter\def\expandafter\da@bars\expandafter{%
      \da@bars
      \dabar@ 
    }%
  }%
  \mathrel{#3}%
  \mathrel{%
    \mathop{\da@bars}\limits
    \ifx\\#1\\%
    \else
      _{\copy0}%
    \fi
    \ifx\\#2\\%
    \else
      ^{\copy2}%
    \fi
  }%
  \mathrel{#4}%
}
\makeatother

\usepackage{geometry}
\newtheorem{thm}{Theorem}[section]
\newtheorem{prop}[thm]{Proposition}
\newtheorem{lem}[thm]{Lemma}
\newtheorem{cor}[thm]{Corollary}

\theoremstyle{definition}
\newtheorem{defn}[thm]{Definition}

\theoremstyle{remark}
\newtheorem{remark}[thm]{Remark}

\numberwithin{equation}{section}

\usepackage{enumerate}
\usepackage{verbatim}
\usepackage{enumitem}

\usepackage[alphabetic]{amsrefs}

\newcommand{\N}{\mathbb{N}}
\newcommand{\F}{\mathbb{F}}
\newcommand{\Z}{\mathbb{Z}}
\newcommand{\Q}{\mathbb{Q}}

\newcommand{\C}{\mathbb{C}}
\renewcommand{\P}{\mathbb{P}}
\newcommand{\A}{\mathcal{A}_g}
\newcommand{\Aut}{\operatorname{Aut}}

\DeclareMathOperator{\Gal}{Gal}

\usepackage{microtype}
\begin{document}

\title{Families of isogenous elliptic curves ordered by height}
\author{Yu Fu}

\maketitle
\begin{abstract}{
Given a family of products of elliptic curves over a rational curve defined over a number field $K$, and assuming that there exists no isogeny between the pair of elliptic curves in the generic fiber, we establish an upper bound for the number of special fibers with height at most $B$ where the two factors are isogenous. Our proof provides an upper bound that is dependent on $K$, the family, and the height $B$. Furthermore, by introducing a slight modification to the definition of the height of the parametrizing family, we prove a uniform bound depends solely on the degree of the family, the field $K$, and $B$. Based on the uniformity, and the fact that the idea of using Heath-Brown type bounds on covers and optimizing the cover to count rational points on specific algebraic families has not been exploited much yet, we hope that the paper serves as a good example to illustrate the strengths of the method and will inspire further exploration and application of these techniques in related research. }
\end{abstract}
\maketitle

\section{Introduction}

Let $k$ be a field. Let $X \to S$ be a family of algebraic varieties over $k$, where $S$ is irreducible, and let $X_{\eta}$ be the generic fiber of this family. People care about what properties of $X_{\eta}$ extend to the special fibers and how we can measure the size of specializations such that a specific property does not extend. For example, the Hilbert Irreducibility Theorem says that for a Galois covering $X \to \mathbb{P}^{n}$ over a number field $K$, for most of the rational points $t \in \mathbb{P}^{n}(K)$ the specializations over $t$ generate a Galois extension with Galois group $G$. Moreover, the size of the complement set, which can be considered as locus of \textit{`exceptional'} points, can be bounded as in \cite{Ser}.

In \cite{EEHK09}, Ellenberg, Elsholtz, Hall, and Kowalski studied families of Jacobians of hyperelliptic curves defined over number fields by affine equations
$$y^{2}=f(x)(x-t), \text{     } t\in \mathbf{A}^{1},$$
 with the assumption that the generic fiber is geometrically simple. They proved that the number of geometrically non-simple fibers in this family, with the height of the parametrizer $t \le B$, is bounded above by a constant $C(f)$ depending on $f$ (\cite[Theorem A]{EEHK09}). Moreover, they obtained an effective bound (\cite[Theorem B]{EEHK09}) that shows how the number of geometrically non-simple fibers grows with $B$
 $$|S(B)| \leqslant C\left(g^2 D(\log 2 B)\right)^{11 g^2}.$$
 Their effective bound comes from sieve methods from analytic number theory and not only depends on the coefficients of $f(x)$ but also depends on the primes dividing the discriminant of $f(x)$ and the genus of the family. In contrast, the uniform bounds for rational points on curves over number fields that come from the determinant method gives stronger uniform bounds than sieve methods. 
 
In this paper, we study families of pairs of elliptic curves defined over a rational curve over a number field $K$. To be precise, let $C$ be a rational curve over $K$ isomorphic to $\mathbb{P}^{1}$ which parametrizes a one-dimensional family of pairs of elliptic curves and let $(E_{\eta}, E^{\prime}_{\eta})$ be the generic fiber of this family over $K(t)$, with the assumption that there exists no isogeny between $E_{\eta}$ and $E^{\prime}_{\eta}$. We establish an upper bound for the number of specializations of $(E_{t}, E^{\prime}_{t})$ where the two curves are isogenous, with the assumption that the height of the parameter $t \le B$. Our method is from a purely arithmetic point of view and relies on constructions and optimizations of explicit covers of the moduli space $X(1) \times X(1)$, combined with results of the uniform bound for points of bounded height on curves, which is completely different from \cite{EEHK09}. Notably, by a proper choice of definition of the height of $t$, the resulting upper bound depends solely on $K$ and $B$ and the degree of the parameter family, as detailed in Theorem \ref{uniform version}. In comparison with the effective bound in loc. cit., the power of the logarithmic term we got is much better. Besides this, the uniformity of our bound adds an exciting dimension to the further exploration of our method.

To discuss the results, we first fix the general setting and terminology (see section $\S 2$ for more details). Let $\iota$ denote the map from $C$ to $\mathbb{P}^{3}$ which is a composition of a finite map via the $j$-invariant, followed by the Segre embedding: 

\begin{equation}\label{iota}
	\iota: C \rightarrow X(1) \times X(1)  \rightarrow \mathbb{P}^{1} \times \mathbb{P}^{1} \hookrightarrow \mathbb{P}^{3}.
	\end{equation}
	To be precise, the map $\iota$ sends $t$ to 

$$(E_{t},E^{\prime}_{t}) \mapsto (j(E_{t}), j(E_{t}^{\prime})) \mapsto (j(E_{t})j(E_{t}^{\prime}); j(E_{t}); j(E_{t}^{\prime}); 1).$$

Degrees and heights are computed with respect to this fixed embedding.  Let $H(\iota)$ be the height of $\iota$ defined by the projective height of the coefficients of the defining polynomials of $j(E_{t})$ and $j(E_{t}^{\prime}).$ Let $H: C(K) \to  \mathbb{R} $ be the projective height defined over $K$ (See section \ref{heights}.1).
We prove the following theorem:

\begin{thm}\label{EC}\label{main}
Let $K$ be a number field of degree $d_K$. Let $C$ be a rational curve over $K$ isomorphic to $\mathbb{P}^{1}$ which parametrizes a one-dimensional family of pairs of elliptic curves $(E, E^{\prime})$. Let $(E_{\eta}, E^{\prime}_{\eta})$ be the generic fiber of this family over $K(t)$, and suppose that there exists no $\overline{K(t)}$-isogeny between $E_{\eta}$ and $E^{\prime}_{eta}$. Let $d=\operatorname{deg} \iota^{*}\mathcal{O}_{\mathbb{P}^3}(1)$ be the degree of the parameter family $C$ defined with respect to $\iota$. Define $S(B)$ to be the set 
$$S(B)=\{t \in C(K) \vert    H(t) \le B, \text{ there is an $\overline{\Q}$-isogeny between } E_{t}  \text{ and } E^{\prime}_{t}\}.$$  
There is an absolute constant $M$ such that for any $B \ge M$, we have 
	
	 $$\vert   S(B)\vert    \lesssim_{K} d^{4}\log(d+1)^4( \log H(\iota) + \log B)^{6}.$$ Here $\lesssim_K$ means less than up to a constant multiplier that depends on $K$.
\end{thm}

\vspace{1em}

	Note that in Theorem \ref{main}, $H(t)$ is the height of $t$ as an element of $\mathbb{P}^{1}_K$. If we modify the definition of the height as follows, namely passing from $\mathbb{P}^{1}$ to $\mathbb{P}^{3}$
    \begin{defn}\label{defniota}
	For a point $P_{t} \in C$ parametrized by $t \in K$, define the height $H(P_{t})$ to be the projective height of $\iota(P_{t}) \in \P^{3}$.\end{defn} 
    and assume that $H(P_t) \le  B$, then we get an \textit{uniform} bound on the number of points $t$ such that $E_t$ and $E^{\prime}_t$ are geometrically isogenous. Moreover, this uniform bound only depends on $K$, the height $B$, and the degree of the family.
	\begin{thm}\label{uniform version}
		Assume the same hypothesis as in Theorem \ref{main}. Let $S^{\prime}(B)$ be the set $$S^{\prime}(B)=\{t \in C(K) \vert    H(P_{t}) \le B, \text{ there is an $\overline{\Q}$-isogeny between } E_{t}  \text{ and } E^{\prime}_{t}\}.$$
		Then we have $$\vert   S^{\prime}(B)\vert  \lesssim_{K} d^{4}(\log B)^{6}.$$	
	\end{thm} 
\subsection*{Relations with Unlikely Intersections}
Although Theorem \ref{EC} indicates the sparsity of isogeny between elliptic curves in a family, one should emphasize that this is not an unlikely intersection problem on its own. There are infinitely many $t \in \overline{\Q}$ such that $E_{t}$ and $E^{\prime}_{t}$ are isogenous! However, since we are working over a fixed number field $K$, this number has to be finite. Nevertheless, one may consider questions in more generalized settings.

 Let $S$ be a GSpin Shimura variety and denote by $\{Z_{i}\}_{i}$ a sequence of special divisors on $S$. Let $C \hookrightarrow S$ be a curve whose generic fiber has \textit{maximal monodromy}. 
 
 \begin{defn}
 Define the set $Z(C)$ to be the intersection of $C$ with the infinite union of the special divisors
 $$Z(C)=C \cap \bigcup_{i} Z_{i}.$$ 
 \end{defn}

 If we take $S=X(1) \times X(1)$ as a special case of GSpin Shimura variety and take the special divisors to be $Z_{n}=Y_{0}(n)$, which parametrizes $n$-isogenous pairs of elliptic curves, then Theorem \ref{uniform version} can be reformulated as follows:
 
 \begin{thm}
 	Let $K$ be a number field of degree $d_K$. Let $C \subset  X(1) \times X(1)$ be a rational curve defined over $K$ parametrizing a family of non-isotrivial and generically non-isogenous elliptic curves. Let $d$ be the projective degree of $C$ defined by $\iota$. For a positive integer $B$ define the set $Z(C;B)$ to be the set of $K$-valued points on $C$ such that
 
 $$Z(C;B)=\{x \in Z(C) \mid H(\iota(x)) \le B\}.$$   
 We have $$|Z(C;B)| \lesssim_{K} d^{4}(\log B)^{6}.$$
 \end{thm}

\vspace{1em}
The discussion above suggests the following question (which we do not claim to know the answer to):
 
 \vspace{1em}
 
 \noindent \textbf{Question:} Suppose $C$ is a curve defined over $\overline{\Q}$ and let $Z(C;B)$ denote the set of $\overline{\Q}$-valued points in $Z(C)$ whose absolute height is bounded above by $B$. 
 One may ask if $Z(C;B)$ is finite. If this is the case, can we get an upper bound in terms of the bounded height $B$?
 
\subsection*{Non-simple abelian varieties in a family}
The method of the proof of Theorem \ref{main} can be used to generalize the results of \cite[Theorem B]{EEHK09}. Let $\mathcal{A}_{g}$ be the moduli space of principally polarized abelian varieties of dimension $g$ and let $n$ be the dimension of $\A$ over a number field $K$. Let $N=n+1$. In recent work, we consider arbitrary families of abelian varieties parametrized by irreducible rational curves over $K$, and prove the following theorem.
\begin{thm}\cite[Theorem 1.4]{FuNonsimple}
	Let $K$ be a number field of degree $d_K$ over $\mathbb{Q}$. Let $X \in \A$ be a rational curve over $K$ parameterizing a one-dimensional family of abelian varieties. Let $A_\eta$ be the generic fiber of this family, and assume that the geometric monodromy of $A_\eta$ modulo $\ell$ is the full symplectic group $\operatorname{Sp}_{2g}(\F_\ell)$ for almost all primes $\ell$. Let $\iota \colon X \to \mathbb{P}^N$ be the restriction of the embedding $\mathcal{A}_g \to \mathbb{P}^N$ defined by the Hodge bundle and let $d$ denote the degree of $X$ with respect to $\iota$. Let $S(B)$ be the non-simple specializations with height of the parameter at most $B$. Then there exists a constant $C_{K,g, \iota}$ depends on $d_K$, $g$ and the family $X$, and an absolute constant $\kappa$ such that $$\vert S(B) \vert  \le C_{K,g,\iota}d^2(\log B)^{2+\kappa}$$ for all $B \ge 1$. 
  
\end{thm}
We note that all upper bounds for the number of non-simple fibers studied in previous literature depend on the explicit defining equation of $X$, and a non-effective upper bound that does not depend on $X$ would follow inexplicitly from Lang's conjecture via the result of Caporaso-Harris-Mazur \cite{CHM}, as explained in \cite{EEHK09}.

 \vspace{1em}

\noindent \textbf{Organization of the paper.} In $\S 2$, we recall the notion of heights on projective spaces, Hecke correspondence, the modular diagonal quotient surfaces, and some results on the Heath-Brown type bounds for curves, which we will use later. In $\S 3$, we interpret the counting problem into counting rational points on projective curves with certain level structures and construct `nice' Galois covers that capture the information of being isogenous. In $\S 4$, we construct certain projective embeddings with respect to the covers in $\S 3$, such that the dimension growth conjecture applies. In $\S 5$, we give an upper bound on the change of heights between covers so that one can bound the height of the lifting points. Finally, we prove Theorem \ref{main} and Theorem \ref{uniform version} in $\S 6$.

\section{Preliminaries}

This section states and proves results and facts in arithmetic geometry, primarily focusing on the theory of heights and the outcomes of dimension growth conjectures for curves. Specifically, we make a survey of uniform bounds for points of bounded height on curves, a crucial ingredient in our proofs, which are utilized in this article. While these concepts are likely familiar to experts, we compile them here for the reader's convenience.

Let $C$ be a rational curve over $K$ isomorphic to $\mathbb{P}^{1}$ parametrizing a one-dimensional family of pairs of elliptic curves, and let $(E_{\eta}, E^{\prime}_{\eta})$ be the generic fiber of this family over $K(t)$, with the assumption that there exists no isogeny between $E_{\eta}$ and $E^{\prime}_{\eta}$. For specializations at $t \in K$, we write $E_{t}$ and $E^{\prime}_{t}$ in the Weierstrass form:

\begin{equation}\label{eq1}
E_{t}: y^{2} = x^{3} + f(t)x + g(t)
\end{equation}

\begin{equation}\label{eq2}
E^{\prime}_{t}: y^{2} = x^{3} + f^{\prime}(t)x + g^{\prime}(t).
\end{equation}

Here, $f(t)$, $g(t)$, $f^{\prime}(t)$, and $g^{\prime}(t)$ are rational functions over $K$. Consequently, the $j$-invariants of $E_{t}$ and $E^{\prime}_{t}$, denoted by $j(E_t)$ and $j(E^{\prime}_t)$, are also rational functions.

\vspace{1em}
\noindent 2.1. \label{heights} \textbf{Heights of points.} For a rational point $a \in K$, define the height $H(a)$ of $a$ to be 
$$H(a):= \prod_{v \in M_{K}} \max\{1, |a|_{v}\}^{\frac{n_{v}}{[K:\Q]}}$$ 
where $n_{v}=[K_{v}:\Q_{v}].$

The height of a $K$-rational point $P \in \P^{n}$ with homogeneous coordinates $P=(x_{0},\cdots, x_{n})$ is defined to be 

$$H(P)=\prod_{v \in M_{K}}\max \{\vert   x_{0}\vert   _{v}, \cdots, \vert   x_{n}\vert   _{v}\}^{\frac{n_{v}}{[K:\Q]}}$$
where $n_{v}=[K_{v}:\Q_{v}].$
 Recall that we define $\iota$ to be the map obtained via the $j$-invariant and via the Segre embedding in $\mathbb{P}^{3}$, see (\ref{iota}).
 
Recall the definition of $H(P_{t})$ in Definition \ref{defniota}. The following lemma indicates that $$H(P_{t})=H(j(E_{t}))H(j(E_{t}^{\prime})).$$
\begin{lem}\label{ProductHeights}
Let $\sigma_{n}$ be the Segre embedding of $n$-copies of $\P^{1}$
$$\sigma_{n}\colon \underbrace{\P^{1} \times \P^{1} \times \ldots \times \P^{1}}_{n\text{-times}} \hookrightarrow \P^{2^{n}-1}.  $$  Let $H(.)$ be the projective height defined above. We have 
$$H(\sigma_{n}(x_1, \cdots, x_{n}))=H(x_{1})\cdots H(x_{n}).$$
	
\end{lem}

\begin{proof}
	By definition of the projective height,
			$$H(\sigma_{n}(x_1, \cdots, x_{n})) = \prod_{v \in M_{K}}\max \{   |\prod_{1 \le i \le n}x_{i}|_{v}, \cdots, |x_{1}|_{v}, \cdots, |x_{n}|_{v}, 1\}^{\frac{n_{v}}{[K:\Q]}}$$
and $$H(x_{1})\cdots H(x_{n})=\prod_{v \in M_{K}} \{\max\{1, |x_{1}|_{v}\} \cdots \max\{1, |x_{n}|_{v}\}\}^{\frac{n_{v}}{[K:\Q]}}.$$
A direct observation shows that for each $v \in M_{K}$, 
$$\max\{ |\prod_{1 \le i \le n}x_{i}|_{v}, \cdots, |x_{1}|_{v}, \cdots, |x_{n}|_{v}, 1 \} = \max\{1, |x_{1}|_{v}\} \cdots \max\{1, |x_{n}|_{v}\}.$$
\end{proof}

\vspace{1em}  

\noindent 2.2. \textbf{The modular diagonal quotient surfaces.} Later in this article, we will define the modular surface $X_{\tilde{H}_{\Delta}}(m)$ (see (\ref{XH})) that lies in a special family, called the \textit{modular diagonal quotient surfaces}, which arise naturally as the (coarse) moduli space to the moduli problem that classifies isomorphisms between mod $m$ Galois representations attached to pairs of elliptic curves $E / K$.

The modular curve $X(m)$ is a Galois cover of $X(1)$ with Galois group $$G=SL_{2}(\Z/m\Z)/\{\pm 1\}.$$ 
Let $\epsilon$ be an element in $(\Z/m\Z)^{\times}$. Let $\alpha_{\epsilon}$ be the automorphism of $G$ defined by conjugation with $Q_{\epsilon}=\big(\begin{smallmatrix}
  \epsilon & 0\\
  0 & 1
\end{smallmatrix}\big)$, i.e. $\alpha_{\epsilon}(g)=Q_{\epsilon}gQ_{\epsilon}^{-1}$. The product surface $X(m) \times X(m)$ carries an action of the twisted diagonal subgroup of $SL_{2}(\Z/m\Z) \times SL_{2}(\Z/m\Z)$ defined by $$\alpha_{\epsilon} \text{ : } \Delta_{\epsilon}=\{(g, \alpha_{\epsilon}(g)): g \in G\}.$$

\begin{defn}
	The twisted diagonal quotient surface defined by $\alpha_{\epsilon}$ is the quotient surface $X_{\epsilon}:= \Delta_{\epsilon} \backslash X(m) \times X(m) $ obtained by the action of $\Delta_{\epsilon}$.
\end{defn}
$X_{\epsilon}$ can be viewed as the moduli space of the triple $(E_1, E_2, Q)$, where 
$$Q: E_{1}[m] \stackrel{\sim}{\rightarrow} E_{2}[m] $$
multiplies the Weil pairing by $\epsilon$.

The modular diagonal quotient surfaces and their modular interpretation are widely used in studying Mazur's question\cite{Maz78} \cite{HK00}, Frey's conjecture\cite{Fre97}, and so on.

\vspace{1em}
\noindent 2.3. \textbf{Uniform bound for points of bounded height on curves.}   

Let $X$ be an irreducible projective curve defined over a number field $K$ with a degree $d$ embedding in $\P^{n}$. Denote by $N(X,B)$ the set of $K$-rational points on $X \subseteq \P^{n}$ of projective heights at most $B$. 
Heath-Brown proved a uniform bound for rational points on a curve $X$ with bounded height in \cite[Theorem 5]{Heath-Brown}, which addresses that 
$$N(X,B) \le O_{\epsilon }(B^{2/d+\epsilon}).$$

The result on irreducible projective curves with the removal of the term $B^{\epsilon}$ without the need for $\log B$ was proved by Walsh in \cite{Walsh}, using a combination of the determinant method based on the $p$-adic approximation introduced by Heath-Brown \cite{Heath-Brown} with the method due to Ellenberg and Venkatesh in \cite{EV05}. The global determinant method, developed by Salberger \cite{Sal23}, was important in removing the term $B^\varepsilon$. The uniform upper bound on $N(X,B)$ with an explicit term of polynomial growth depends on $d$ was proved by Castryck, Cluckers, Dittmann and Nguyen \cite{CCDN19}, based on this technique. 

\begin{thm}\cite[Theorem 2]{CCDN19}\label{CCDN}
	Given $n \ge 1$, there exists a constant $c=c(n)$ such that for all $d > 0$ and all integral projective curves $X \in \P^{n}_{\Q}$ of degree $d$ and all $B \ge 1$ one has
	
	$$\vert   N(X,B)\vert    \le cd^{4}B^{2/d}.$$
	
\end{thm}

A year later, Paredes and Sasyk \cite{PS22} extended the work of Castryck, Cluckers, Dittmann, and Nguyen to give uniform estimates for the number of rational points of bounded height on projective varieties defined over global fields. More precisely, they proved the following extension of \cite[Theorem 2]{CCDN19} to global fields.

\begin{thm}\cite[Theorem 1.8]{PS22} \label{PS}
Let $K$ be a global field of degree $d_K$. Let $H$ be the absolute projective multiplicative height. For any integral projective curve $C \subseteq \mathbb{P}_K^N$ of degree $d$ it holds
$$
|\{\boldsymbol{x} \in C(K): H(\boldsymbol{x}) \leq B\}| \lesssim_{K, N} \begin{cases}d^4 B^{\frac{2 d_K}{d}} & \text { if } K \text { is a number field, } \\ d^8 B^{\frac{2 d_K}{d}} & \text { if } K \text { is a function field. }\end{cases}
$$
\end{thm}

We will discuss this result in section $\S 6$, where we apply it.

\section{Construct Galois Covering with Level Structures}

In this section, we construct `nice' Galois coverings of $X(1) \times X(1)$ with level structures. To be precise, for a large enough rational prime $m$, these are quotients of $X(m) \times X(m)$ by certain kinds of subgroups of $SL_{2}(\Z/m\Z) \times SL_{2}(\Z/m\Z)$, with the property that one can lift a point $t \in S(B)$ to one of the quotients. In other words, the Galois coverings capture points that parametrize isogenous pairs  $(E_{t}, E^{\prime}_{t})$ in the family. The main result of this section is Lemma \ref{covers}. 

\vspace{1em} 

We assume that $m$ is a rational prime and let $X(m) \times X(m)$ be the surface parametrizing $4$-tuples $(E, E^{\prime},\phi,\phi^{\prime})$, where $E_{t}$ and $E^{\prime}_{t}$ are elliptic curves and $\phi$, $\phi^{\prime}$ are $m$-level structures, i.e. 

$$\phi: E[m] \to (\Z/m\Z)^{2}, \phi^{\prime}: E^{\prime}[m] \to (\Z /m\Z)^{2}$$
are isomorphisms of group schemes preserving the Weil pairing. Let $H$ be a subgroup of $SL_{2}(\Z/m\Z) \times SL_{2}(\Z/m\Z)$ and define $X_{H}$ to be the quotient 

\begin{equation}\label{XH}
	X_{H}:= (X(m) \times X(m))/H.
\end{equation}

\vspace{1em}
\noindent \textbf{Mod $p$ monodromy representations and the lifting criterion.} Let $E$ be an elliptic curve over a number field $K$ and let $p$ be a prime. Denote by 

$$\rho_{E,p}: G_{K} \to \Aut(E[p])$$
the mod $p$ Galois representation associated to the $p$-torsion points $E[p]$ of the elliptic curve $E$. It is a standard fact that the elliptic curve $E/ K$ admits an isogeny of degree $p$ defined over $K$ if and only if the image $\rho_{E, p}\left(G_{K}\right)$ is contained in a Borel subgroup of $\Aut(E[p])$. If $E/ K$ and $E^{\prime} / K$ are related by an isogeny over $K$ of degree coprime to $p$, then this isogeny induces a $G_{K}$-module isomorphism $E[p] \simeq E^{\prime}[p]$, which identifies the images $\rho_{E, p}\left(G_{K}\right)$ and $\rho_{E^{\prime}, p}(G_{K})$ up to change of basis. See \cite{RV00} for an explicit description of the images of mod $p$ Galois representations attached to the product of two isogenous elliptic curves with an isogeny of degree $p$. 

\begin{defn}\label{Hdelta}
	Define $H_{\Delta}$ to be the image of the diagonal map 
	
	$$\Delta: SL_{2}(\Z/m\Z) \to SL_{2}(\Z/m\Z) \times SL_{2}(\Z/m\Z).$$
\end{defn}

We will prove, in Proposition \ref{Qbar}, that there exists an isogeny $\phi_{t}: E_{t} \to E^{\prime}_{t}$ defined over $\overline{\Q}$ if and only if the monodromy of the mod $p$ Galois representation on $E_{t}[p] \times E^{\prime}_{t}[p]$ is contained in a group $\tilde{H}_{\Delta}$, which contains $H_{\Delta}$ as an index $2$ subgroup for all but finitely many $p$.  

 \vspace{1em}
\begin{remark}
  It is a classical result that for elliptic curves defined over a number field $K$, there are finitely many $j$-invariants with complex multiplication in $K$. Denote this number by $CM(K)$. We need to bound the number of points $P_t$ on $C$ whose $j$-invariant lies in this set. There are at most $CM(K)^{2}$ points $P_{t} \in C$ such that $\iota(P_t)=(j(E_t), j(E^{\prime}_t))$ has one or more coordinates being CM. Therefore in our setting, we can discard them and focus on pairs of elliptic curves without complex multiplication.
  \end{remark}

\begin{lem}\label{iso}
	Let $E_{1}$ and $E_{2}$ be two elliptic curves without complex multiplication over a number field $K$. If there exists an isogeny $\phi: E_{1} \to E_{2}$ defined over $K$, then the mod $p$ Galois representation of $E_{1} \times E_{2}$
	$$\Gal(\bar{K}/K) \to GL_{2}(\F_{p}) \times GL_{2}(\F_{p})$$ 
	has image conjugate to $H_{\Delta}$ for primes $p$ not dividing the degree of $\phi$. Conversely, if the mod $p$ Galois representation of $E_{1} \times E_{2}$ has image conjugate to $H_{\Delta}$ for infinitely many primes $p$, then $E_1$ is isogenous to $E_2$ over $K.$
	
\end{lem}
 \begin{proof}
 
 First, by Serre's open image theorem, the mod $p$ Galois representation of each factor
 $$\Gal(\bar{K}/K) \to GL_{2}(\F_{p})$$
 is surjective for all large enough primes $p$.
  
 Suppose $E_{1}/K$ and $E_{2} / K$ are related by an isogeny over $K$ of degree $d$, then for all primes $p \nmid d$ this isogeny induces a $G_{K}$-module isomorphism from $E_{1}[p]$ to $E_{2}[p]$, which identifies the images $\rho_{E_{1}, p}\left(G_{K}\right)$ and $\rho_{E_{2}, p}(G_{K})$. 
 
 Now suppose that for infinitely many primes $p$ the $\bmod$ $p$ Galois representation of $E_1 \times E_2$ has diagonal image, i.e. $E_1[p] \simeq E_2[p]$ as $G_K$ modules. Hence for any $\ell$ there are infinitely many $p$ such that $a_{\ell}\left(E_1\right) \equiv a_{\ell}\left(E_2\right) \text{ } \bmod $ $p$. Then we have $a_{\ell}\left(E_1\right)=a_{\ell}\left(E_2\right)$. By Faltings' theorem on isogenies, two elliptic curves over $K$ are $K$-isogenous if and only if their traces of $\bmod$ $\ell$ Frobenius agrees at all primes $\ell$ with good reduction. The lemma follows.   

\end{proof}
 
 Note that Lemma \ref{iso} is a result over number fields and will serve as a main ingredient in the proof of Proposition \ref{Qbar}. However, later in the proof of Proposition \ref{genericE}, we need to use a similar result over the function field $K(t)$ which classifies elliptic curves up to isogeny by their mod $p$ Galois representations. Here we present a beautiful theorem by Bakker and Tsimerman \cite[Theorem 1]{BT}.
 \begin{thm}\label{Bakker and Tsimerman}
 	
 Let $k$ be an algebraically closed field of characteristic $0$. For any $N > 0$, there exists $M_N > 0$ such that for any prime $p > M_N$ and any smooth quasi-projective curve $U$ of gonality $n<N$, non-isotrivial elliptic curves $\mathcal{E}$ over $U$ are classified up to isogeny by their $p$-torsion local system $\mathcal{E}[p]$.
 \end{thm}

The following lemma, together with the fact that if $F$ is a field with more than $5$ elements then the only proper normal subgroup of $SL_{2}(F)$ is the group $\{\pm 1\}$, proves that $\tilde{H}_{\Delta}$ is the unique proper subgroup of $SL_{2}(\Z/m\Z) \times SL_{2}(\Z/m\Z)$ that contains $H_{\Delta}$ as an index $2$ subgroup.

\begin{lem}
Let $A$ be a group and let $G=A \times A$. Define $\Delta=\{(a,a) \mid a \in A\}$ as the diagonal subgroup of $G$. If $\Delta \le H \le G$ then there exists a normal subgroup $N$ of $A$ such that $H=\{(g,h) \in G \mid gh^{-1} \in N \}$.	
\end{lem}

\begin{proof}
	Let $N=\{h \in A \mid (h,1) \in H \}$ be a subgroup of $G$. We claim that $N$ is the desired normal subgroup. Indeed, for any $a \in A$ and $(h,1) \in H$, we have $(aha^{-1},1)=(a,a)(h,1)(a^{-1}, a^{-1}) \in H$, therefore $N$ is a normal subgroup of $A$.
	
	For any $a, a^{\prime} \in A$, we have $(aa^{\prime -1},1)(a^{\prime}, a^{\prime})=(a, a^{\prime})$. Therefore $(a, a^{\prime}) \in H$ if and only if $(aa^{\prime -1},1) \in H$, if and only if $aa'^{-1}\in N$ by the definition of $N$.
\end{proof}

We will need the following defintion:
\begin{defn}
Let $\tilde{H}_{\Delta}$ be the maximal subgroup of $SL_{2}(\Z/m\Z) \times SL_{2}(\Z/m\Z)$ that contains $H_{\Delta}$ as an index $2$ subgroup
$$\tilde{H}_\Delta=\{(g,\pm g)\colon g\in SL_2(\Z/m\Z)\}.$$
\end{defn}
\begin{prop}\label{Qbar}
	Let $E_{1}$ and $E_{2}$ be elliptic curves without complex multiplication over $\Q$. There exists an isogeny $\phi: E_{1} \to E_{2}$ defined over $\overline{\Q}$ if and only if the mod $p$ Galois representation 
	$$\Gal(\overline{\Q}/\Q) \to GL_{2}(\F_{p}) \times GL_{2}(\F_{p})$$ has image contains $\tilde{H}_{\Delta}$, for all primes $p$ not dividing the degree of $\phi$. 
\end{prop}

\begin{proof}
We need the following lemma:
	 \begin{lem}\label{twist}
Let $E_{1}$, $E_{2}$ be elliptic curves without complex multiplication defined over a number field $K$. If $E_{1}$ and $E_{2}$ are isogenous over $\overline{\Q}$ then there exists a quadratic twist of $E_{2}$ that is isogenous to $E_{1}$ over $K$.	
\end{lem}
Suppose there exists an isogeny $\varphi: E_{1} \to E_{2}$ over $\overline{\Q}$ then by Lemma \ref{twist} and Lemma \ref{iso}, there is a quadratic extension $L/K$ such that $G_{L}$ has diagonal image in $GL_{2}(\F_{p}) \times GL_{2}(\F_{p})$. Therefore the image of $G_{K}$ is contained in a subgroup of $GL_{2}(\F_{p}) \times GL_{2}(\F_{p})$ which contains $H_{\Delta}$ as an index $2$ subgroup.
For the other direction, we take the preimage of the $H_{\Delta}$, which can be written in the form $G_{F}$ for some number field $F$, that is quadratic over $K$. By Lemma \ref{iso}, there is an isogeny $\varphi: E_{1} \to E_{2}$ defined over $F$ which completes the proof. 

\end{proof}

  \begin{proof}[Proof of Lemma \ref{twist}]
  Let $\varphi: E_{1} \to E_{2}$ be an isogeny over $\overline{\Q}$ and $G_{\overline{\Q}/K}=\Gal(\overline{\Q}/K)$. For every $g\in G_{\overline{\Q}/K}$, $\varphi^{g}=[\alpha(g)] \circ \varphi $ is another isogeny $E_{1} \to E_{2}$ of the same degree as $\varphi$. Here we define $\alpha: G_{\overline{\Q}/K} \to \mathbb{R}$ be a character on $G_{\overline{\Q}/K}$. Since for all elliptic curves without complex multiplication over $\overline{\Q}$, there exists a cyclic isogeny $\varphi: E_{1} \to E_{2}$ up to sign and all other isogenies $\psi$ from $E_{1}$ to $E_{2}$ can be written as $\psi=\varphi \circ [m]$ for some integer $m$, we have $\alpha(g)= \pm 1$. Hence $\alpha(g)$ is a quadratic character and there exists $d \in K^{*}$ such that $\alpha(g)=g(\sqrt{d})/\sqrt{d}$. Thus the quadratic twist of $E_{2}$ by $d$ is the desired twist. 
 
  \end{proof}

\vspace{1em}

\begin{defn}\label{Hprod}
  Let $H_{\text{p}}:=H_{1} \times H_{2}$, where $H_{1}$ and $H_{2}$ (possibly $H_{1}=H_{2}$) are maximal parabolic subgroups of $SL_{2}(\Z/m\Z)$.
\end{defn}

\begin{remark}

All maximal parabolic subgroups of $SL_{2}(\Z/m\Z)$ are in the same conjugacy class. Therefore the covers constructed in this way are all isomorphic to each other. If $t \in S(B)$ lifts to one, it lifts to all.

\end{remark}

\begin{lem}\label{covers}
	
	Any rational point $t \in S(B)$ for some $B$, where the specializations $E_t$ and $E^{\prime}_t$ above $t$ are non-CM, admits a lifting to one of the congruence covers: $X_{\tilde{H}_{\Delta}}(K)$ or $X_{H_{\text{p}}}(K)$.
\end{lem}

\begin{proof}
	By the argument above, when there is a $K$-isogeny of degree $m \nmid d$, it induces a $G_{K}$-isomorphism $E_{t}[m] \simeq E^{\prime}_{t}[m]$ which implies an isomorphism of the mod $m$ Galois images, up to a change of basis given by conjugating by an element $(1,g) \in SL_{2}(\Z/m\Z)$. Applying Proposition \ref{Qbar}, we get the conclusion that if $m$ does not divide the degree of the isogeny, the Galois image lies in $\tilde{H}_{\Delta}$. 	
	As for the rest of the lemma, suppose there is an isogeny $\phi: E_t \to E^{\prime}_t$n with $\deg \phi =m$. For any Galois element $g$, we have $\phi^{g}=[\pm 1]\circ \phi$ and the kernel of $\phi$ is rational. Then we have a $K$-rational subgroup of order $m$ in both $E_t[m]$ and $E^{\prime}_t[m]$. Therefore the Galois image of the $m$-torsion monodromy representation of both $E_{1}$ and $E_{2}$ contained in a Borel subgroup of $SL_{2}(\Z/m\Z)$. Since any Borel subgroup is maximal parabolic in $SL_{2}(\Z/m\Z)$, this proves the lemma \ref{covers}.
\end{proof}

\section{Geometric Interpretation and Projective Embeddings}

In this section, we explore the geometric interpretation of the covers in Lemma \ref{covers} and construct projective embeddings, which allow us to transform the problem of counting rational points (with bounded height) on $C$ into counting rational points (with bounded height) in projective spaces. 

\begin{defn}
Let $H$ be a subgroup of $SL_{2}(\Z/m\Z) \times SL_{2}(\Z/m\Z)$ such that $H$ is either $\tilde{H}_{\Delta}$ or $H_{\operatorname{p}}$. Define $C_{H}$ to be the lifting of $C$ to the modular surface $X_{H}:= X(m) \times X(m)/H$, which is given by the pullback of the following diagram.

$$\begin{tikzcd}\label{dia1}
C_{H} \arrow[r] \arrow[d] &  X_{H}  \arrow[d]  \\
C \arrow[r] & X(1) \times X(1)  	
\end{tikzcd}$$

\end{defn}
\subsection{The lift curve is integral}
In order to apply the result by Paredes and Sasyk, we need to prove that the lifting $C_{H}$ of $C$ is integral for large enough $m$. We need the following Goursat's lemma:
	  
	   \begin{lem}\cite[Theorem 5.5.1]{Hall}\label{Hall}
	   Let $G$ and $H$ be groups, and let $K$ be a subdirect product of $G$ and $H$; that is, $K \leq G \times H$, and $\pi_G(K)=G, \pi_H(K)=H$, where $\pi_G$ and $\pi_H$ are the projections onto the first and second factor, respectively from $G \times H$. Let $N_1=K \cap \operatorname{ker}\left(\pi_G\right)$ and $N_2=K \cap \operatorname{ker}\left(\pi_H\right)$. Then $N_2$ can be identified with a normal subgroup $N_G$ of $G, N_1$ can be identified with a normal subgroup $N_H$ of $H$, and the image of $K$ in $G / N_G \times H / N_H$ is the graph of an isomorphism $G / N_G \cong H / N_H$.
\end{lem}

Also, we prove the following proposition:

 \begin{prop}\label{choice of m}
	For all $m \ge 17$, the Galois image of the $m$-torsion monodromy representation
$$G_{\overline{K}(t)}\rightarrow \Aut(E_\eta[m])$$
    of $E_{\eta}[m]$ and $E^{\prime}_{\eta}[m]$ is $SL_{2}(\Z/m\Z)$.
\end{prop}
\begin{proof}
	Suppose the Galois image of the $m$-torsion monodromy representation is some proper subgroup $G$ of $SL_{2}(\Z/m\Z)$. Then we have a dominant map $f: C \to X(m)/G$. Since the genus of $C$ is zero, this implies that the genus of the modular curve $X(m)/G$ is zero. 
	
	Let $\mathcal{N}(m)$ be the quantity such that
	
	$$\mathcal{N}(m):= \operatorname{min} \{ \text{genus}(X(m)/G) \mid G \subsetneq SL_{2}(\Z/m\Z),\text{ } G \text{ maximal}\}.$$
	Cojocaru and Hall proved the genus formula for $X(m)/G$ for all possible maximal subgroup $G$ of $SL_{2}(\Z/m\Z)$, which is summarized in a table \cite[Table 2.1]{CH05}. The number $\mathcal{N}(m)$ also coincides when the minimum is taken over all proper subgroups of $SL_2(\mathbb{Z}/m\mathbb{Z})$. Moreover, they proved that 
	$$\mathcal{N}(m)=\frac{1}{12}\left[m-\left(6+3 \mathrm{e}_2+4 \mathrm{e}_3\right)\right]>0$$
	for $m \ge 17$. The proposition follows.
\end{proof}

 Let $U \subset C$ be the connected dense open locus parametrizing smooth points, i.e., pairs of genuine elliptic curves. The \'etale fundamental group $\pi_{1}(U)$ is a quotient of the absolute Galois group of $K(t)$, which acts on the $m$-torsion of the generic fiber, say $E_{t}[m]$ and $E^{\prime}_{t}[m]$, in the usual Galois way.

\begin{lem}\label{full monodromy} 
 Let $M^{\prime}=\max\{17,M_1\}$ and $m \ge M^{\prime}$, where the absolute integer $M_1$ is defined in Theorem \ref{Bakker and Tsimerman}. Then the Galois image of the $m$-torsion monodromy representation of $\pi_{1}(U)$ is the full group $SL_{2}(\Z/m\Z) \times SL_{2}(\Z/m\Z)$.
  \end{lem}

  \begin{proof} Proposition \ref{choice of m} asserts that for $m \ge 17$, the map
  $$\rho_{m}: \pi_{1}(U) \to SL_{2}(\Z/m\Z) \times SL_{2}(\Z/m\Z)$$
  is surjective on each of the two factors. Therefore the reduction map
  $$\bar{\rho}_{m}: \pi_{1}(U) \to PSL_{2}(\Z/m\Z) \times PSL_{2}(\Z/m\Z)$$
  is surjective on each factor. For $m \ge 5$, the projective special linear group $$PSL_{2}(\Z/m\Z)=SL_{2}(\Z/m\Z)/\{\pm 1\}$$ is simple. By assumption, the generic fiber $E_{t}$ and $E^{\prime}_{t}$ of the family are not isogenous. Therefore by the work of Bakker and Tsimerman \cite[Theorem 1]{BT}, there exists an absolute constant $M_1$ such that for any prime $m > M_{1}$, the image of $\rho_{m}$ on each of the factors is non-isomorphic. We claim that the image of the reduction map $\bar{\rho}_{m}$ also has non-isomorphic factors. Hence the inequivalence condition in Lemma \ref{Hall} is satisfied. Lemma \ref{Hall} leads to the conclusion that the Galois image of the $m$-torsion monodromy representation of $\pi_{1}(U)$ is full in $PSL_{2}(\Z/m\Z) \times PSL_{2}(\Z/m\Z)$.
  
  We now prove the claim. It is sufficient to show that for any isomorphism $\nu: PSL_2(\Z/m\Z) \to PSL_2(\Z/m\Z)$ there exists a unique isomorphism $ \upsilon: SL_2(\Z/m\Z) \to SL_2(\Z/m\Z)$ lifts $\nu$. The existence of such lift is due to the vanishing of $H^2(SL_2(\F_m), \mu_2)$ for any prime $m \ge 5.$ Indeed, since $H_2(SL_2(\F_m), \Z)=0$, the universal coefficient theorem implies $$H^2(SL_2(\F_m), \mu_2) \simeq \operatorname{Ext}^1(H_1(SL_2(\F_m),\Z), \mu_2)$$ and $$H_1(SL_2(\F_m),\Z)=SL_2(\F_m)/[SL_2(\F_m),SL_2(\F_m)]=0.$$
  
  Take the matrix \[
   S=
  \left[ {\begin{array}{cc}
   1 & 0 \\
   1 & 1 \\
  \end{array} } \right]
\]
in $SL_2(\Z/m\Z)$. Assume that the image of $S$ under the composition of the reduction map $SL_2(\Z/m\Z) \to PSL_2(\Z/m\Z)$ and $\nu$, lift to $A$ or $-A$ in $SL_2(\Z/m\Z)$. Since $S^m=I_2$, the lifted homomorphism, if it exists, must send $S$ to whichever $A$ or $-A$ whose $m$-th power is the identity matrix. Similarly for the matrix \[
   T=
  \left[ {\begin{array}{cc}
   1 & a \\
   0 & 1 \\
  \end{array} } \right]
\]
where $a$ is some primitive element in $\F_m^*$. By Dickson's theorem, $S$ and $T$ generate $SL_2(\Z/m\Z)$. This proves the uniqueness of the lift.
  \end{proof}
  
Now we are ready to prove that $C_{H}$ is irreducible. This follows as a consequence of Proposition \ref{choice of m}: 

\begin{lem}\label{connected}
For each choice of $H$, the curve $C_{H}$ is integral.	
\end{lem}
\begin{proof}
	We have the restriction map $q_{\tilde{U}}: \tilde{U} \to U$ where $\tilde{U}$ is the preimage of $U$ under the quotient map $q$. By Lemma \ref{full monodromy}, the monodromy of $\pi_{1}(U)$ acts transitively on the right cosets $SL_{2}(\Z/m\Z) \times SL_{2}(\Z/m\Z) / H$, which implies that the cover $C_{H}$ is connected for both $H= \tilde{H}_{\Delta}$ and $H=H_{\operatorname{p}}$. The lemma follows by the fact that the quotient space of a connected space is connected. 
  
\end{proof}

 \subsection{Construct projective embeddings}
In this subsection, we give an explicit construction of projective embeddings of $C_{H}$, denoted by $\iota_{H}$. We make the following diagram commute for each case's choice of $N \in \Z$. Here the rational map $\P^{N} \dashrightarrow \P^{3}$ is a projection of coordinates of $\P^N$.
 
\begin{equation}\label{embedding}
\begin{tikzcd}
 C_{H}   \arrow[r, hook, "\iota_{H}"] \arrow[d, "q", swap] & \P^{N} \arrow[d, dotted]   \\
 C   \arrow[r, hook] & \P^{3}. 	
\end{tikzcd} 
\end{equation}

  \noindent \textbf{Case I: The modular diagonal quotient surfaces.}
  
  Recall the definition of the modular diagonal quotient surfaces in $\S 2.2$. In our case where $\epsilon=1$, $X_{H_{\Delta}}(m)(K)$ has the moduli interpretation that it is the set of isomorphism classes of triples $\left(E_{1}, E_{2}, \psi\right)$, where $E_{1}$, $E_{2}$ are elliptic curves over $K$ and $\psi: E_{1}[m] \stackrel{\sim}{\rightarrow} E_{2}[m]$ is an isomorphism of the $m$-torsion subgroups of the elliptic curves which preserves the Weil pairing. 
Let $t$ be a rational point of $C(K)$ such that there exists a point $(E_{t},E_{t}^{\prime}, \psi: E_{t}[m] \stackrel{\sim}{\rightarrow} E_{t}^{\prime}[m])$ which is a point of $X_{H_{\Delta}}(m)(K)$. 

We now define some functions on $C_{\tilde{H}_{\Delta}}$ in order to apply the result from \cite{CCDN19} and to bound the number of points on $C_{\tilde{H}_{\Delta}}$.

For a fixed $t \in K$, there is a list of elliptic curves isogenous to $E_{t}$ through cyclic isogenies of degree $m$ given by the list of cyclic subgroups of $E_{t}[m]$, say

$$E_{t,1},\cdots, E_{t,m+1}.$$

Similarly we have a list of $m$-cyclic subgroups of $E_{t}^{\prime}[m]$ parametrizing the $m$-cyclic isogenies of $E_{t}^{\prime}$, with the corresponding list of elliptic curves isogenous to $E_{t}^{\prime}$: 

$$E_{t,1}^{\prime},\cdots, E_{t, m+1}^{\prime}.$$

\begin{remark}
	These isogenies will most likely \textit{not} be defined over $K$ unless $E_{t}[m]$ or $E_{t}^{\prime}[m]$ has a rational cyclic subgroup. Even then, most of them will not be defined over $K$.
\end{remark}

For each point $(E_{t}, E_{t}^{\prime}, \psi)$ on $C_{\tilde{H}_\Delta}$, we have $m+1$ cyclic subgroups of $E_{t}[m]$ and $m+1$ cyclic subgroups of $E_{t}^{\prime}[m]$ which can be placed in natural bijection with each other under $\psi$. We can re-order the lists for $E_{t}^{\prime}$ such that $E_{t,1}$ is in correspondence with $E_{t,1}^{\prime}$ and so on.
\vspace{1em}
 
\begin{defn}\label{F}

	Let $F$ be a function defined on $X_{\tilde{H}_{\Delta}}$ given by
	$$F(E_{t}, E_{t}^{\prime}, \psi)=j(E_{t,1})j(E^{\prime}_{t,1})+ \cdots +j(E_{t, m+1})j(E^{\prime}_{t, m+1}).$$
	
	\end{defn}
\vspace{1em}
\begin{lem}

$F$ is defined over $\Q$. Moreover, $F$ is an element of the function field $\Q(X_{\tilde{H}_{\Delta}}).$	
\end{lem}

\begin{proof}

	For any $g \in \Gal(\overline{\Q}/\Q)$, $t \in \overline{\Q}$, denote by $(t, \psi)$ one of the preimages on $X_{\tilde{H}_{\Delta}}$. 
	$$F((t,\psi)^{g})=\Sigma_{i=1}^{m+1} j(E_{t^{g},i})j(E^{\prime}_{t^{g},i})= \Sigma_{i=1}^{m+1} (j(E_{t,i})j(E^{\prime}_{t,i}))^{g}.$$
	The second equality holds since $\psi^{g}=\psi$ and for an arbitrary $1 \le i \le m+1$, there is a unique $k$ such that 
	$$(j(E_{t,i}))^{g}=j(E_{t,k})$$
	and once we fix $i$, it is also true that  
	$$(j(E_{t,i}^{\prime}))^{g}=j(E_{t,k}^{\prime})$$ 
	for the same $k$.

\end{proof}

 \begin{defn}
 	Define $\iota_{\tilde{H}_{\Delta}}^{\prime}: C_{\tilde{H}_{\Delta}} \to \P^{1} \times \P^{1} \times \P^{1}$ to be the function given by 	 $$\iota_{\tilde{H}_{\Delta}}^{\prime}((E_{t},E_{t}^{\prime}, \psi)) = (F(E_{t}, E_{t}^{\prime}, \psi), j(E_{t}), j(E_{t}^{\prime})).$$

 \end{defn}
\begin{prop}\label{genericE}
For $m \ge 17$, $\iota^{\prime}_{\tilde{H}_{\Delta}}$ is generically an embedding of $C_{\tilde{H}_{\Delta}}$ into $\P^{1} \times \P^{1} \times \P^{1}$. This is equivalent to saying that the subfield $M \subset K(C_{\tilde{H}_{\Delta}})$ generated by $j(E_{t})$, $j(E_{t}^{\prime})$ and $F$ is the whole function field.
	
\end{prop}
\vspace{1em}
The proof of Proposition \ref{genericE} splits into two parts. The first part is to show that $M$ is not contained in $K(C)$. The second part is to show there is no intermediate extension between $K(C)$ and $K(C_{\tilde{H}_{\Delta}})$, so that $K(C) \subseteq M \subseteq K(C_{\tilde{H}_{\Delta}})$ and $M \ne K(C)$ implies $M=K(C_{\tilde{H}_{\Delta}}).$

\vspace{1em}

First, we present some representation theoretical lemma which we will use later.
Let $G$ be a finite group and let $V$ and $V^{\prime}$ be permutation representations of $G$ of the same dimension, such that $G$ acts $2$-transitively on some finite sets of $V$ and $V^{\prime}$. It is a standard fact that the corresponding representation $V$(resp. $V^{\prime}$) is the direct sum of the trivial representation and an irreducible representation whose coordinates sum to $0$. Denote the irreducible representation by $V_{0}$(resp. $V^{\prime}_{0}$). We prove that if $v \in V$, $v^{\prime} \in V^{\prime}$ such that $\langle v, v^{\prime} \rangle = \langle v, (v^{\prime})^{g} \rangle$ for any $g \in G$, then either $v$ or $v^{\prime}$ is fixed by $G$. We first prove the following lemma.

\begin{lem}\label{4.6}
If $v \in V$(resp. $v^{\prime} \in V^{\prime}$) is not in the trivial representation generated by $[1, \cdots , 1]^{T}$, the set $\{v-v^{g}\}$(resp. $\{v^{\prime}-(v^{\prime})^{g}\}$), as $g$ ranges over $G$, span the space of all vectors in $V_{0}$(resp. $V^{\prime}_{0}$). 
\end{lem}

\begin{proof}
	Let $W_{v}$(resp. $W^{\prime}_{v}$) be the subspace Span$\{v-v^{g}\}$(resp. Span$\{(v^{\prime})-(v^{\prime})^{g}\}$) as $g$ ranges over $G$. We claim that $W_{v}$ is a sub-representation of $V_{0}$, the irreducible representation of the permutation representation given by the condition that the sum of the coordinates equals $0$. Indeed, for any $v \in W_{v}$, $\sigma, \tau \in G$, one have
	$$(v-v^{\sigma})^{\tau}=v^{\tau}-v^{\sigma\tau}=(v-v^{\sigma\tau})-(v-v^{\tau}).$$
The lemma follows from the face that $V_{0}$(resp. $V^{\prime}_{0}$) is irreducible.
\end{proof}

\begin{lem}\label{sigma}
	If $v \in V$, $v^{\prime} \in V^{\prime}$ such that $\langle v, v^{\prime} \rangle = \langle v, (v^{\prime})^{g} \rangle$ for any $g \in G$, then either $v$ or $v^{\prime}$ is fixed by $G$.
\end{lem}

\begin{proof}
	If $\langle v, v^{\prime} \rangle = \langle v, (v^{\prime})^{g} \rangle$ for all $g \in G$, then 
	 $$\langle v, (v^{\prime}) - (v^{\prime})^{\gamma} \rangle =0$$
	 for all $g \in G.$
	
	If $W_{v^{\prime}}=0$, then $v^{\prime}= (v^{\prime})^{g}$ for any $g \in G$ thus $v^{\prime}$ is fixed by $G$.
	If $W_{v^{\prime}} \ne 0$, by lemma \ref{4.6}, $v^{\prime}- (v^{\prime})^{g}$ span the space $V^{\prime}_{0}$ which is the orthogonal complement of the trivial representation. Therefore $$v \in W_{v^{\prime}}^{\perp}= \mathbb{C}\langle [1, \cdots, 1]^{T} \rangle,$$
	thus fixed by $G$.
	
\end{proof}

\vspace{1em}

Now we apply Lemma \ref{sigma} to the case where $G=SL_{2}(\Z/m\Z)$ and $m$ is a prime. Here $V$ and $V^{\prime}$ are $m+1$ dimensional complex representations of $SL_{2}(\F_m)$ spanned by vectors $v_t$ and $v_{t}^{\prime}$ respectively, as $t$ ranges over $C$. For $t \in C$ such that $E_{t}$ and $E_{t}^{\prime}$ are both non-singular, define $v_t$ and $v_{t}^{\prime}$ by 

\begin{equation}
	v_t =(j(E_{t,1}),\cdots,j(E_{t, m+1}))
\end{equation}

\begin{equation}
	v_{t}^{\prime} = (j(E^{\prime}_{t,1}),\cdots,j(E^{\prime}_{t, m+1})).
\end{equation}
It is easy to see that $V$ and $V^{\prime}$ are permutation representations of $SL_{2}(\Z/m\Z)$ by its transitive action on the basis.

Recall the definition of $F$ in Definition \ref{F}. We may also write $F$ as an inner product:  
$$F(E_{t}, E_{t}^{\prime}, \psi)= \langle v_t , v^{\prime}_{t} \rangle = j(E_{t,1})j(E^{\prime}_{t,1})+ \cdots +j(E_{t, m+1})j(E^{\prime}_{t, m+1}).$$
Note that $F$ is defined on $C_{\tilde{H}_{\Delta}}$ by restriction, but $SL_{2}(\Z/m\Z) \times SL_{2}(\Z/m\Z)$ does not act on $C_{\tilde{H}_{\Delta}}$. By Lemma \ref{connected} $C_{\tilde{H}_{\Delta}}$ is connected, we define $C^{\prime}$ as the Galois closure of $C_{\tilde{H}_{\Delta}}$ over $C$. This is equivalent to saying that $C^{\prime}$ is the pullback of $C$ all the way to $X(m) \times X(m)$, see diagram (\ref{pullback}). Therefore the pullback of $F$ to $X(m) \times X(m)$ is a function on $C^{\prime}$ with $F^{\sigma}=F$ for all $\sigma \in \tilde{H}_{\Delta}$. We prove that $F$ is not defined on $C$.

\begin{equation}\label{pullback}
\begin{tikzcd}
 C^{\prime} \arrow[r,hook] \arrow[d] & X(m) \times X(m) \arrow[d] \\
 C_{\tilde{H}_{\Delta}} \arrow[d] \arrow[r,hook] & X(m) \times X(m)/ \tilde{H}_{\Delta} \arrow[d] \\
 C  \arrow[r,hook] & X(1) \times X(1). 	
\end{tikzcd} 
\end{equation}

\vspace{1em}
 \begin{lem}\label{F is not defined on C}
 	$F$ is not defined on $C$.
 \end{lem}
 
 \begin{proof}
 For each $\gamma \in SL_{2}(\Z/m\Z)$, the action of $\gamma$ on $F$ is given by  
 
$$F((E_{t}, E_{t}^{\prime}, \psi)^{\gamma})= \langle v_t , (v^{\prime}_{t})^{\gamma} \rangle.$$
If $F$ is defined on $C$, then for every $\gamma \in SL_{2}(\Z/m\Z)$ we have $F^{\gamma}=F$ therefore $\langle v_{t}, v_{t}^{\prime} \rangle = \langle v_t, (v_{t}^{\prime})^{\gamma} \rangle$. By Lemma \ref{sigma}, either $v_{t}$ or $v^{\prime}_{t}$ is fixed by $SL_{2}(\Z/m\Z)$. But this implies either
$$j(E_{t,1})=\cdots=j(E_{t, m+1})$$ 
or
$$j(E^{\prime}_{t,1})=\cdots=j(E^{\prime}_{t, m+1}).$$ 
Since there exists $t \in C$ such that $E_{t}$ and $E_{t}^{\prime}$ are both non-CM elliptic curves, this cannot happen. Therefore $F$ cannot be defined on $C$. 

 \end{proof}

 \vspace{1em}
	  Now we are ready to prove the generic injectivity of $\iota^{\prime}_{H_{\Delta}}$. We show no intermediate cover exists between $C_{\tilde{H}_\Delta}$ and $C$ by an argument using monodromy. 	 
  
  \begin{proof}[\textit{Proof of Proposition \ref{genericE}}]
  By lemma \ref{sigma}, we have the argument that $F$ is not defined on $C$.
  
  Recall that $C_{\tilde{H}_{\Delta}}$ is defined to be the cover of $C$ constructed by pulling back $C \to X(1) \times X(1)$ in the previous context. We prove that there is no intermediate cover between $C_{\tilde{H}_{\Delta}}$ and $C$. Suppose there is a curve $X$ such that 
  $$C_{\tilde{H}_{\Delta}} \to X \to C$$
  with all maps of degrees greater than $1$. Since the connected covering space of $C$ is in bijection with the subgroups of $\pi_{1}(C)$, $X$ corresponds to a proper subgroup $H^{\prime}$ strictly containing $\tilde{H}_{\Delta}$. However, $\tilde{H}_{\Delta}$ is a maximal subgroup of $PSL_{2}(\Z/m\Z) \times PSL_{2}(\Z/m\Z)$, which implies that $X$ is isomorphic to $C_{\tilde{H}_{\Delta}}$ by Lemma \ref{full monodromy}, contradiction.
  Therefore $C_{\tilde{H}_{\Delta}}$ is birational to its image under $\iota_{\tilde{H}_{\Delta}}^{\prime}$, which proves the proposition.    	 	
 	  \end{proof}
 	  
\vspace{1em} 	  
We have constructed a generic embedding $\iota^{\prime}_{\tilde{H}_{\Delta}}$ of $C_{\tilde{H}_{\Delta}}$ into a product of projective lines from the argument above. Composing with the Segre embedding, we get a generic embedding of $C_{\tilde{H}_{\Delta}}$ into $\P^{7}$, denoted by $\iota_{\tilde{H}_{\Delta}}$.

$$\iota_{\tilde{H}_{\Delta}}: C_{\tilde{H}_{\Delta}} \hookrightarrow \P^{1} \times \P^{1} \times \P^{1} \hookrightarrow \P^{7}.$$

One notice that $\iota_{\tilde{H}_{\Delta}}$ fits into the diagram \ref{embedding} at the beginning of this chapter, with $N=7$. 

\vspace{1em}
\noindent \textbf{Case II: the maximal parabolic quotient surfaces.}
   When $H=H_{\operatorname{p}}$ which is a product of maximal parabolic subgroups, the quotient $X(m) \times X(m)/H$ is isomorphic to $X_{0}(m) \times X_{0}(m)$. The corresponding projective embedding $\iota_{H_{\operatorname{p}}}$ can be constructed in the context of Hecke correspondence of level $SL_{2}(\Z/m\Z)$ followed by the Segre embedding, as following:
   
\begin{equation}\label{iotaHp}
	\iota_{H_{\operatorname{p}}}: X_{0}(m) \times X_{0}(m) \hookrightarrow \P^{1} \times \P^{1} \times \P^{1} \times \P^{1} \hookrightarrow \P^{15}.
\end{equation}

To be explicit, for a point $\tilde{P_{t}}=(E_{t},\tilde{E_{t}}, E_{t}^{\prime}, \tilde{E_{t}^{\prime}})$ that lifts $P_{t}=(E_{t},E_{t}^{\prime})$, where $E_{t}$ and $\tilde{E_{t}}$ are linked by a cyclic isogeny of degree $m$ and same for $E_{t}^{\prime}$ and $\tilde{E_{t}^{\prime}}$, one may write $\iota_{H_{\operatorname{p}}}$ as 

\begin{align*}
\iota_{H_{\operatorname{p}}}: (E_{t}, \tilde{E_{t}}, E_{t}^{\prime}, \tilde{E_{t}^{\prime}}) 
	 &\mapsto  j(E_{t}) \times j(\tilde{E_{t}}) \times j(E_{t}^{\prime}) \times j(\tilde{E_{t}^{\prime}}) \\
	 &\mapsto (j(E_{t})j(\tilde{E_{t}})j(E_{t}^{\prime})j(\tilde{E_{t}^{\prime}}); \cdots; j(E_{t}); j(\tilde{E_{t}}); j(E_{t}^{\prime}); j(\tilde{E_{t}^{\prime}}); 1).
\end{align*}
In this case, we choose $N=15$ in diagram \ref{embedding}.

  \section{Bound for the Change of Heights}
  
  In this section, we give an upper bound on the height of a point in $C_{H}(K)$ lying over a point $P_{t} \in C(K)$, in terms of the height $H(t)$ and the level $m$. The main result of this section is Proposition \ref{change of heights}.
  
\vspace{1em}
\noindent 5.1. \textbf{Hecke correspondence and modular polynomials.} Modular polynomials of elliptic curves, the so-called 'elliptic modular polynomials,' are the most common and simplest examples of modular equations. For a positive integer $m$, the classical modular polynomial $\Phi_{m}$ is the minimal polynomial of $j(mz)$ over $\C(j)$. In other words we have $\Phi_{m}(j(mz),j(z))=0$. The bivariate polynomial $\Phi_{m}(X,Y)$ is symmetric of degree $\psi(m)=m\prod_{p\vert   m}(1+p^{-1})$ in both variables, and its coefficients grow super-exponentially in $m$. The modular curve $Y_{0}(m)$ is birational to its image in $\P^{1} \times \P^{1}$ with $\Phi_{m}$ an equation for this image. The graph of $\Phi_{m}$ describes the Hecke correspondence such that there exists a cyclic isogeny of degree $m$ between projections onto each copy of $\P^{1}$. 

\vspace{1em}
For elliptic curves $E_{1}$ and $E_{2}$ linked with a cyclic isogeny of order $m$, we aim to find an upper bound for the height $H(j(E_{1}))$, in terms of $H(j(E_{2}))$ and the coefficients of $\Phi_{m}$. This has been worked out by Pazuki in \cite{Paz19}. 

\begin{thm}{\cite[Theorem 1.1]{Paz19}}\label{paz}
	Let $\varphi: E_{1} \rightarrow E_{2}$ be a $\overline{\mathbb{Q}}$-isogeny between two elliptic curves defined over $\overline{\mathbb{Q}}$. Let $j_{1}$ and $j_{2}$ be the respective $j$-invariants. Then one has
	
		$$\left\vert   h\left(j_{1}\right)-h\left(j_{2}\right)\right\vert    \leq 9.204+12 \log \operatorname{deg} \varphi$$ 
		where $h(.)$ denotes the absolute logarithmic Weil height.

\end{thm}

Theorem \ref{paz} leads to the following corollary:

\begin{cor}\label{pazc}
Let $\varphi: E_{1} \rightarrow E_{2}$ be a $\overline{\mathbb{Q}}$-isogeny between two elliptic curves defined over $\overline{\Q}$ which is cyclic of degree $m$. Let $j_{1}$ and $j_{2}$ be the respective $j$-invariants. Then one has
	$$H(j_{1}) < Am^{12}H(j_{2})$$
	for some absolute constant $A$. Here $H(.)$ denotes the projective height.
\end{cor}

\vspace{1em}
\noindent 5.2. \textbf{Bounding change of heights.} We prove an upper bound on the product of heights, which we will use later.

\begin{lem}\label{5.3}
Let $d$ be the projective degree of $C$ under the embedding $\iota$, see (\ref{iota}). Let $H(\iota)$ be the height of $\iota$ defined by the height of the coefficients of the defining polynomials of $j(E_{t})$ and $j(E_{t}^{\prime}).$ Then for every $t \in K$ with $H(t) \le B$, we have
$$H(P_{t}) \le (d+1)H(\iota)B^{d}.$$	
\end{lem}

\begin{proof}
By Lemma \ref{ProductHeights} we have
$$H(P_{t})=H(j(E_{t}))H(j(E_{t}^{\prime})).$$
The lemma then follows from \cite[VIII, Theorem 5.6]{Sil09}
, which asserts that when there is a map of degree $d$ between two projective spaces, say 
$$\lambda: \mathbb{P}^{m} \to \mathbb{P}^{M},$$
then for all points $P \in P^{m}(\overline{\Q})$
there are positive constant $C_{1}$ and $C_{2}$ depending on $F$ such that
$$C_{1}H(P)^{d} \le H(\lambda(P)) \le C_{2}H(P)^{d}.$$

Write $\lambda=[\lambda_{0}, \cdots, \lambda_{M}]$ using homogeneous polynomials $\lambda_{i}$ having no common zeros. Let $H(\lambda)$ be the height of $\lambda$ defined by the height of the coefficients of $\lambda_{i}$. The constant $C_{1}$ and $C_{2}$ can be explicitly calculated in terms of $M$, $m$ and $H(\lambda)$. Especially, we can let $C_{1}= \binom{m+d}{m} H(\lambda).$ The lemma follows from the assumption that $\lambda=\iota$, $m=1$ and $M=3$.
 
\end{proof}

\vspace{1em}
 By Lemma \ref{covers}, for our choice of $m \in \Z $, a rational point $t \in C(K)$ with $t \in S(B)$ for some $B$ lifts to a rational point on one of the covers $C_{\tilde{H}_{\Delta}} \subset X_{\tilde{H}_{\Delta}}$ or $C_{H_{\operatorname{p}}} \subset X_{H_{\operatorname{p}}}$. We have the following proposition:

\begin{prop}\label{change of heights}

Fix $m \in \N$. Let $t \in K$ be a rational point such that $t \in S(B)$ for some $B$ and the specializations $E_t$ and $E^{\prime}_t$ above $t$ are non-CM elliptic curves. Let $P_{t}$ denote the point on $C$ parametrized by $t$, and denote by $\tilde{P_{t}}$ a lifting of $P_{t}$ to one of the covers $C_H$ in Lemma \ref{covers}. Let $H(\iota_{H}(\tilde{P_{t}}))$ denote the projective height of $\tilde{P_{t}}$ with respect to $\iota_{H}$. 

If  $H=\tilde{H}_{\Delta}$, then 

 $$H(\iota_{H}(\tilde{P_{t}})) \le (m+1)m^{24(m+1)}A^{2(m+1)}((d+1)H(\iota))^{m+2}B^{d(m+2)} .$$

 If $H=H_{\operatorname{p}}$, then 

 $$H(\iota_{H}(\tilde{P_{t}})) \le  m^{24}(d+1)^{2}H(\iota)^{2}B^{2d} .$$

\end{prop}

\begin{proof}
	\textbf{Case 1: $P_{t}$ lifts to $C_{\tilde{H}_{\Delta}}(K)$}

	As noted above, working on the quotient $C_{\tilde{H}_{\Delta}}$ we have $\iota_{\tilde{H}_{\Delta}}$ embeds $C_{\tilde{H}_{\Delta}}$ into $\P^{7}$ by composing $\iota^{\prime}_{\tilde{H}_{\Delta}}$ with the Segre embedding. A point $\tilde{P_{t}}=(E_{t},E_{t}^{\prime}, \psi: E_{t}[m] \stackrel{\sim}{\rightarrow} E_{t}^{\prime}[m])$ that is a lift of $P_{t}=(E_{t},E_{t}^{\prime})$ embedded into $\P^{7}$ as following:
	
	$$(E_{t},E_{t}^{\prime}, \psi) \to  F \times j(E_{t}) \times j(E_{t}^{\prime}) \hookrightarrow [Fj(E_{t})j(E_{t}^{\prime}), Fj(E_{t}), Fj(E_{t}^{\prime}),\cdots, 1].$$
	
	Since our ultimate goal is to count rational points parametrized by $t \in K$, we need a formula relating the height of $\iota_{H}(\tilde{P_{t}})$ with heights of $F$, $j(E_{t})$ and $j(E_{t}^{\prime})$. By Lemma \ref{ProductHeights}, we have
	\begin{equation}
	\label{3}
		H(\iota_{H}(\tilde{P_{t}}))=H(F)H(j(E_{t}))H(j(E_{t}^{\prime})).
	\end{equation}
	Let $i$ be an integer between $1$ and $m+1$ such that 
	$$ H(j(E_{t,i})j(E_{t,i}^{\prime}))= \max_{1 \le k \le m+1} H(j(E_{t,k})j(E_{t,k}^{\prime})).$$
	  By Definition \ref{F} and Corollary \ref{pazc} and Lemma \ref{5.3}, together with the fact that for any $\alpha, \beta, \alpha_{1} \cdots \alpha_{r} \in \overline{\Q}$,
	  $$H(\alpha\beta) \le H(\alpha)H(\beta)$$
	  and $$H(\alpha_{1}+\cdots\alpha_{r}) \le rH(\alpha_{1})\cdots H(\alpha_{r}),$$
	  we have 
	
	\begin{align}
	H(F) & =H(j(E_{t,1})j(E^{\prime}_{t,1})+ \cdots +j(E_{t, m+1})j(E^{\prime}_{t, m+1})) \\
	& \le (m+1) H(j(E_{t,i}))^{m+1}H(j(E_{t,i}^{\prime}))^{m+1} \\
	& \le (m+1)m^{24(m+1)}A^{2(m+1)}H(j(E_{t}))^{m+1}H(j(E_{t}^{\prime}))^{m+1}  \\
	& \le (m+1)m^{24(m+1)}A^{2(m+1)}((d+1)H(\iota))^{m+1}B^{d(m+1)} .
\end{align}
The constant $A$ comes from Corollary \ref{pazc}, and the first part of the lemma follows from $(\ref{3})$.

\vspace{1em}
\noindent \textbf{Case 2: $P_{t}$ lifts to one of the maximal parabolic quotient surfaces}	

Recall the definition of $H_{\operatorname{p}}$ (\ref{Hprod}) and $\iota_{H_{\operatorname{p}}}$ (\ref{iotaHp}). As in the previous case, Lemma \ref{ProductHeights} implies that 
 
 $$H(\iota_{H_{\operatorname{p}}}(\tilde{P_{t}}))=H(j(E_{t}))H(j(\tilde{E_{t}}))H(j(E_{t}^{\prime}))H(j(\tilde{E_{t}^{\prime}}))$$
 By Corollary \ref{pazc} and Lemma \ref{5.3}, where $\tilde{E}_{t}$ is $m$-isogenous to $E_{t}$ and $\tilde{E}^{\prime}_{t}$ is $m$-isogenous to $E_{t}^{\prime}$, we have 

$$H(\iota_{H_{\operatorname{p}}}(\tilde{P_{t}})) \ll m^{24}(d+1)^{2}H(\iota)^{2}B^{2d} .$$

\end{proof}

\section{Proof of the Main Theorems}
 \subsection{Proof of Theorem \ref{main}}
The previous sections show that for a rational point $P_{t}$ on $C$, we have two types of possible liftings to some modular surfaces with $m$-level structures. Accordingly, we divide the proof of Theorem \ref{main} into two parts and analyze each part's contribution to $\vert  S(B)\vert $. We make optimization of $m$ in terms of the height $B$ of $t$ as.

\textbf{Case 1: Contributions from modular diagonal quotient surfaces. }

Recall that we have the following commutative diagram: 
		$$\begin{tikzcd}
   C_{\tilde{H}_{\Delta}}  \arrow[r, hook, "\iota_{\tilde{H}_{\Delta}}^{\prime}"] \arrow[d, "q", swap] & \P^{1} \times \P^{1} \times \P^{1}  \arrow[r, hook, "\text{Segre}"] \arrow[d] & \P^{7}  \\
C \arrow[r, hook] & \P^{1} \times \P^{1} \arrow[r, hook] & \P^{3}  	
\end{tikzcd}$$

Let $\iota_{\tilde{H}_{\Delta}}$ be the composition of $\iota^{\prime}_{\tilde{H}_{\Delta}}$ with the Segre embedding. In order to apply Theorem \ref{PS}, we bound the degree of $\iota_{\tilde{H}_{\Delta}}$, which depends on $m$ and the projective degree of $C$.

Let $\deg_{C_{\tilde{H}_{\Delta}}}(F)$ be the degree of $F$ as a function on $C_{\tilde{H}_{\Delta}}$. The degree of the function $\iota_{\tilde{H}_{\Delta}}^{\prime}$ on $C_{\tilde{H}_{\Delta}}$ can be viewed as a tridegree, which we denote by $(\deg_{C_{\tilde{H}_{\Delta}}}(F),e, e^{\prime})$. When we pass to $\P^{7}$ by compose with the Segre embedding, we have $$\iota_{\tilde{H}_{\Delta}}=\deg_{C_{\tilde{H}_\Delta}}(F)+e+e^{\prime}.$$ Here $e$ denotes the degree of the function $j(E_{t})$ on $C_{\tilde{H}_{\Delta}}$ and $e^{\prime}$ is the degree of $j(E^{\prime}_{t})$ on  $C_{\tilde{H}_{\Delta}}$. Let $\alpha$ be the degree of the cover $q$ and let $d_{E}$ be the degree of the $j$-invariant map $C \to \P^{1}$. Therefore $e=\alpha d_{E}$. Similarly, we have $e^{\prime}=\alpha d_{E^{\prime}}.$ Therefore $$\deg(\iota_{\tilde{H}_{\Delta}}) > e+e^{\prime}=\alpha(d_{E}+d_{E^{\prime}})=\alpha d.$$    

The degree of $q$ is equal to the index of $\tilde{H}_{\Delta}$ inside the Galois group $G=SL_{2}(\Z/m\Z) \times SL_{2}(\Z/m\Z)$. We have 

$$\alpha = [G:\tilde{H}_{\Delta}] = m^{3}(1-\frac{1}{m^{2}})=\frac{1}{2}m(m+1)(m-1).$$	
Therefore
$$\operatorname{deg}(\iota_{\tilde{H}_{\Delta}}) \ge \frac{1}{2}m(m+1)(m-1)d.$$

In order to get an upper bound of $\deg(\iota_{\tilde{H}_{\Delta}})$, we need an upper bound for the degree of $F$ over $C_{\tilde{H}_{\Delta}}$. Recall from diagram \ref{pullback}, $C^{\prime}$ is the Galois closure of $C_{\tilde{H}_{\Delta}}$ over $C$ and the function $F$ is defined to be 

$$F=j(E_{t,1})j(E^{\prime}_{t,1})+ \cdots +j(E_{t, m+1})j(E^{\prime}_{t, m+1}).$$

The individual terms $j(E_{t,i})$ and $j(E^{\prime}_{t,i})$ are defined on $C^{\prime}$, instead of on $C_{\tilde{H}_{\Delta}}$. A point on $C^{\prime}$ is a product of triples $(E_{t}, P, N) \times (E^{\prime}_{t}, P^{\prime}, N^{\prime})$, where $P$(resp. $P^{\prime}$) is a point of $E_{t}[m]$(resp. $E_{t}^{\prime}[m]$) and $N$(resp. $N^{\prime}$) is a cyclic subgroup of $E_{t}[m]$(resp. $E_{t}^{\prime}[m]$). Fix a $j$-invariant $x$, the degree of $j(E_{t,i})$(resp. $j(E^{\prime}_{t,i})$) is the number of points on $C^{\prime}$ such that $j(E_{t}/N)=x$. As long as $E_t$ and $E_{t}^{\prime})$ is not CM, there are $m+1$ $j$-invariants which are $m$-isogenous to $x$ and there are $d_{E}$ points on $C$ mapping to each of those $(m+1)$ $j$-invariants. Hence the degree $\deg_{C_{\tilde{H}_\Delta}}(j(E_{t,i}))$(resp. $\deg_{C_{\tilde{H}_\Delta}}(j(E^{\prime}_{t,i}))$) for each $1 \le i \le m+1$ is $(m+1)d_{E}$(resp. $(m+1)d_{E^{\prime}}$).

The argument above, together with the fact that if $f$ and $g$ are functions on a curve $X$ then
$$\deg(f+g) \le \deg(f) + \deg(g)$$
and $$\deg(fg) \le \deg(f) +\deg(g),$$
yields that

$$\deg_{C_{\tilde{H}_{\Delta}}}(F) \le \deg_{C^{\prime}}(F) \le (m+1)^{2}d. $$
Hence we get an upper bound on the degree of $\iota_{\tilde{H}_{\Delta}}$ which is 
$$\operatorname{deg}(\iota_{\tilde{H}_{\Delta}}) \le (m(m+1)(m-1) + (m+1)^{2})d.$$ 

We state the above
result in a lemma:

\begin{lem}\label{bound of the degree of iota}
	$$\frac{1}{2}m(m+1)(m-1)d \le \deg(\iota_{\tilde{H}_{\Delta}}) \le (m(m+1)(m-1) + (m+1)^{2})d.$$
\end{lem}

\vspace{1em}
Let $S_{B, m, H_{\Delta}}$ be the set of rational points on $C_{\tilde{H}_{\Delta}}(K)$ which are liftings of $P_{t}$ for some $t \in S(B)$. Recall that we prove an upper bound for the heights of points in $S_{B, m, H_{\Delta}}$ in Proposition \ref{change of heights}. Theorem \ref{PS} (\cite[Theorem 1.8]{PS22}) then applies, along with Lemma \ref{connected} and Lemma \ref{bound of the degree of iota}, yielding
\begin{align}
	|S_{B, m, H_{\Delta}}| & \lesssim_{K} ((\alpha + (m+1)^{2}) d)^{4}((m+1)m^{24(m+1)}A^{2(m+1)}((d+1)H(\iota))^{m+2}B^{d(m+2)})^{\frac{2d_{K}}{\alpha d}}  \\
	 & \lesssim_{K} ((m^{2}+1)(m+1)d)^{4}((m+1)m^{24(m+1)}A^{2(m+1)}((d+1)H(\iota))^{m+2}B^{d(m+2)})^{\frac{2d_{K}}{m(m-1)(m+1) d}} \\
	 & \lesssim_{K} (m^{3}d)^{4}(m+1)^{\frac{2d_K}{m(m-1)(m+1) d}}A^{\frac{4d_{K}}{m(m-1)d}}m^{\frac{48d_{K}}{m(m-1)d}}((d+1)H(\iota)B)^{\frac{2d_{K}(m+2)}{m(m-1)(m+1)}}\label{6.3}
	 \end{align}

     Let $M_{1}$ be the absolute integer defined in Theorem \ref{Bakker and Tsimerman}. Let $m \ge M^{\prime}=\operatorname{max}\{17, M_{1}\}$. 
	 The terms in (\ref{6.3}) other than $(m^{3}d)^{4}$ are bounded above by an absolute constant. The argument requires optimizations on the choice of $m$, which we will prove in the following lemma.
	 
	 \begin{lem}\label{optimization}
	 For a fixed $B$, let $m$ be a prime between $2(\log (d+1) + \log H(\iota) + \log B)^{\frac{1}{2}}$ and $4(\log (d+1) + \log H(\iota) + \log B)^{\frac{1}{2}}$. Then there is an absolute constant $A_{0}$ such that 
	 $$(m+1)^{\frac{2d_{K}}{m(m-1)(m+1) d}}A^{\frac{4d_{K}}{m(m-1)d}}m^{\frac{48d_{K}}{m(m-1)d}}((d+1)H(\iota)B)^{\frac{2d_{K}(m+2)}{m(m-1)(m+1)}} \le e^{A_{0}d_K}.$$	
	 \end{lem}
	 
  \begin{proof}
  Once we write $$(m+1)^{\frac{2d_{K}}{m(m-1)(m+1) d}}$$ as $$e^{(\frac{2d_{K}}{m(m-1)(m+1) d})\log (m+1)},$$ it is easy to see that for $m \ge 2$ we have
$$(m+1)^{\frac{2d_{K}}{m(m-1)(m+1) d}} \ll e^{\frac{2d_{K}}{d}} \le e^{2d_{K}}.$$
This is because $\frac{\log (m+1)}{m(m-1)(m+1)} $ is bounded above by $1$.
A similar argument shows that
$$m^{\frac{48d_{K}}{m(m-1)d}} \ll e^{\frac{48d_{K}}{d}} \le e^{48d_{K}}$$ which also contributes as a constant independent of $m$ and $d$.

It is left to consider $((d+1)H(\iota)B)^{\frac{2d_{K}(m+2)}{m(m-1)(m+1)}}$ which plays an important role in the optimization process. We make the optimization by choosing suitable $m$ in terms of $B$, $d$, and $H(\iota)$. The following inequalities, together with Proposition \ref{choice of m}, allows one to take $m$ to be any prime between $2(\log (d+1) + \log H(\iota) + \log B)^{\frac{1}{2}}$ and $4(\log (d+1) + \log H(\iota) + \log B)^{\frac{1}{2}}$, so that

\begin{align*}
((d+1)H(\iota)B)^{\frac{2d_{K}(m+2)}{m(m-1)(m+1)}} &=e^{{\frac{2d_{K}(m+2)}{m(m-1)(m+1)}}(\log (d+1) + \log H(\iota) + \log B)} \\
&= e^{\frac{2d_{K} (\log (d+1) + \log H(\iota) + \log B)}{(m-1)(m+1)}} \cdot e^{\frac{4d_{K} (\log (d+1) + \log H(\iota) + \log B)}{m(m-1)(m+1)}} \\
&\le e^{N_{1}d_{K}} \cdot e^{N_{2}d_{K}}
\end{align*}
where $\frac{2 (\log (d+1) + \log H(\iota) + \log B)}{(m-1)(m+1)d}$ and $\frac{4 (\log (d+1) + \log H(\iota) + \log B)}{m(m-1)(m+1)d}$ are bounded above by some absolute constant $N_{1}$ and $N_{2}$. 

  \end{proof}
  
 The conditions on $m$ imply that we need to assume $4(\log (d+1) + \log H(\iota) + \log B)^{\frac{1}{2}} \ge M^{\prime}$ and thus we can always assume that $B$ is greater than the absolute integer $M=e^{(\frac{M^{\prime}}{2})^2}$ to make everything go through.
 
We have the following proposition as a conclusion of the case.
\begin{prop}\label{1}
The number of points in $|S(B)|$ that comes from the modular diagonal quotient surface is bounded above by 
\begin{align*}
	|S_{B, m, H_{\Delta}}| &\lesssim_{K} d^{4}((\log (d+1) + \log H(\iota) + \log B))^{6} \\
	&\lesssim_{K} d^{4+\epsilon}( \log H(\iota) + \log B)^{6} .
\end{align*} 
	
\end{prop}  
  \begin{proof} 
  The proposition follows from the inequality (\ref{6.3}) and Lemma \ref{optimization}.
\end{proof}

\vspace{2em}	
\noindent \textbf{Case 2: Contributions from maximal parabolic quotient surfaces. }
Recall that in this case, we have the following commutative diagram.
$$\begin{tikzcd}
   C_{H_{\operatorname{p}}}  \arrow[r, hook, "\iota^{\prime}_{{H}_{\operatorname{p}}}"] \arrow[d, "q", swap]  & \P^{1} \times \P^{1} \times \P^{1} \times \P^{1}  \arrow[r, hook, "\text{Segre}"] \arrow[d] & \P^{15}   \\
C  \arrow[r, hook] & \P^{1} \times \P^{1} \arrow[r, hook] & \P^{3} 	
\end{tikzcd}$$

The degree of the covering space $q$, denoted by $\beta$, is equal to the index of $H_{\operatorname{p}}$ inside the Galois group $G$, similar to the previous case. The index of $H_{\operatorname{p}}$ as a product of maximal parabolic subgroups is equal to 

$$\beta = [G:H_{\operatorname{p}}] = \frac{1}{4}(m+1)^{2}.$$
One notice that the degree of $\iota_{H_{\operatorname{p}}}$ satisfies $\operatorname{deg}(\iota_{{H}_{\operatorname{p}}})=\beta d$. Let $S_{B, m, H_{\operatorname{p}}}$ be the set of rational points in $C_{H_{\operatorname{p}}}(K)$ that lift $P_{t}$ for some $t \in S(B)$. Applying Proposition \ref{change of heights} and Theorem \ref{PS}, we get the following inequality

$$|S_{B, m, H_{\operatorname{p}}}| \lesssim_{K} ((m+1)^{2}d)^{4}(m^{24}(d+1)^{2}H(\iota)^{2}B^{2d})^{\frac{8d_{K}}{(m+1)^{2}d}}$$
which allows us to make the same optimization as in case 1. By taking $$m \sim (\log B + \log (d+1) + \log H(\iota))^{\frac{1}{2}},$$ we get an upper bound of the total contribution to $\vert   S(B)\vert $ from the maximal parabolic surfaces: 

\begin{prop}\label{2}

	\begin{align*}
		|S_{B, m, H_{\operatorname{p}}}| &\lesssim_{K} d^{4}(\log B + \log (d+1) + \log H(\iota))^{4} \\
&\lesssim_{K} d^{4+\epsilon}(\log B  + \log H(\iota))^{4}.
	\end{align*}

\end{prop}

 \begin{proof}[Proof of Theorem \ref{main}]
 
 It is easy to see that the contribution from case $1$ dominates that of case 2. Theorem \ref{main} then follows from Proposition \ref{1} and \ref{2}.

 \end{proof}
	
\subsection{Proof of Theorem \ref{uniform version}}

Recall that in Theorem \ref{main}, $H(t)$ is defined as the height of $t$ as an element of $K$. Instead, we calculate the height of $t$ as the height of a point on $X(1) \times X(1)$, which we have been calling $H(P_{t})$. Assume that $H(P_t) \le  B$. In this section, we prove an \textit{uniform} bound on the number of points $t$ such that $E_t$ and $E^{\prime}_t$ are geometrically isogenous, which only depends on $K$, $B$, and the \textit{degree} of the parametrize family.

We need a slightly modified version of Proposition \ref{change of heights}, which we state as a corollary.
\begin{cor}\label{corollary of change of heights}
	Fix $m \in \N$. Let $t \in K$ be a rational point such that $t \in S(B)$ for some $B$. Let $P_{t}$ denote the point on $C$ parametrized by $t$, and denote by $\tilde{P_{t}}$ a lifting of $P_{t}$ to one of the covers $C_H$ in Lemma \ref{covers}. Let $H(\iota_{H}(\tilde{P_{t}}))$ denote the projective height of $\tilde{P_{t}}$ with respect to $\iota_{H}$. 

If $H=\tilde{H}_{\Delta}$, then 

 $$H(\iota_{H}(\tilde{P_{t}})) \le (m+1)m^{24(m+1)}A^{2(m+1)}B^{(m+2)} .$$

 If $H=H_{\operatorname{p}}$, then 

 $$H(\iota_{H}(\tilde{P_{t}})) \le  m^{24}B^{2} .$$

\end{cor}

\begin{proof}
	The proof is the same as Proposition \ref{change of heights}, except that we have 
	$$H(j(E_{t}))H(j(E_{t}^{\prime})) \le B$$
	instead of Lemma \ref{5.3}.
\end{proof}

\begin{proof}[Proof of Theorem \ref{uniform version}]
	Let $S^{\prime}_{B,m,H}$ be the set of rational points on $C_{H}(K)$ which are preimages of $P_t$ for some $t \in S^{\prime}(B).$ A similar argument to the proof of Theorem \ref{main} yields 
	\begin{align}
	|S^{\prime}_{B, m, \tilde{H}_{\Delta}}| & \lesssim_{K} ((\alpha + (m+1)^{2}) d)^{4}((m+1)m^{24(m+1)}A^{2(m+1)}B)^{(m+2)\frac{2d_{K}}{\alpha d}}  \\
	 & \lesssim_{K} ((m^{2}+1)(m+1)d)^{4}((m+1)m^{24(m+1)}A^{2(m+1)}B^{(m+2)})^{\frac{2d_{K}}{m(m-1)(m+1)d}}\\
	 & \lesssim_{K} (m^{3}d)^{4}(m+1)^{\frac{2d_K}{m(m-1)(m+1) d}}A^{\frac{4d_{K}}{m(m-1)d}}m^{\frac{48d_{K}}{m(m-1)d}}B^{\frac{2d_{K}(m+2)}{m(m-1)(m+1)d}}\label{6.6}
	 \end{align} 
	 
	 when $H=\tilde{H}_{\Delta},$
	 and
	 \begin{equation}\label{6.7}
	 	|S^{\prime}_{B, m, H_{\operatorname{p}}}| \lesssim_{K} ((m+1)^{2}d)^{4}(m^{24}B^{2})^{\frac{8d_{K}}{(m+1)^{2}d}}.
	 \end{equation}
	
	We choose $m$ to be a prime between $2(\log B)^{\frac{1}{2}}$ and $4(\log B)^{\frac{1}{2}}$ as to control the growth of the $B$-power factors in both \ref{6.6} and \ref{6.7}, such that 
		
	\begin{align*}
B^{\frac{2d_{K}(m+2)}{m(m-1)(m+1)d}} &=e^{\frac{2d_{K}(m+2)}{m(m-1)(m+1)d} \log B} \\
&= e^{\frac{2d_{K} (\log B)}{(m-1)(m+1)}} \cdot e^{\frac{4d_{K} (\log B)}{m(m-1)(m+1)}} \\
&\le e^{N^{\prime}_{1}d_{K}} \cdot e^{N^{\prime}_{2}d_{K}}
\end{align*}
and  
$$
B^{\frac{16d_{K}}{(m+1)^{2}d}} = e^{\frac{16d_{K}}{(m+1)^{2}d} \log B} \le e^{N^{\prime}_{3}d_{K}}
$$
for some absolute constant $N^{\prime}_{1}$, $N^{\prime}_{2}$ and $N^{\prime}_{3}$. 
	\end{proof}

\section*{Acknowledgement} The author wishes to thank Jordan Ellenberg for many helpful suggestions. Also, The author wants to show her gratitude to Asvin G., Ananth Shankar and Xiaoheng (Jerry) Wang for the useful conversations and comments on the earlier version of the draft, and Chunhui Liu for pointing out the reference \cite{PS22} to her which helps sharpen the bound. We thank the referees for the valuable feedback and comments.

\section*{Compliance with ethical standards}
\subsection*{Conflicts of interest}
  The author states that there are no conflicts of interest.

\subsection*{Data availability} Data sharing not applicable to this article as no data sets were generated or analyzed.

\bibliography{references.bib}

\begin{bibdiv}
\begin{biblist}

\bib{BT}{article}{
      author={Bakker, Benjamin},
      author={Tsimerman, Jacob},
       title={p-torsion monodromy representations of elliptic curves over
  geometric function fields},
        date={2016},
        ISSN={0003486X},
     journal={Annals of Mathematics},
      volume={184},
      number={3},
       pages={709\ndash 744},
         url={http://www.jstor.org/stable/44072028},
}

\bib{CCDN19}{article}{
      author={Castryck, Wouter},
      author={Cluckers, Raf},
      author={Dittmann, Philip},
      author={Nguyen, Kien~Huu},
       title={{The dimension growth conjecture, polynomial in the degree and
  without logarithmic factors}},
        date={2020},
     journal={Algebra \& Number Theory},
      volume={14},
      number={8},
       pages={2261 \ndash  2294},
         url={https://doi.org/10.2140/ant.2020.14.2261},
}

\bib{CH05}{article}{
      author={Cojocaru, Alina~Carmen},
      author={Hall, Chris},
       title={{Uniform results for Serre's theorem for elliptic curves}},
        date={2005},
        ISSN={1073-7928},
     journal={International Mathematics Research Notices},
      volume={2005},
      number={50},
       pages={3065\ndash 3080},
  eprint={https://academic.oup.com/imrn/article-pdf/2005/50/3065/2045809/2005-50-3065.pdf},
         url={https://doi.org/10.1155/IMRN.2005.3065},
}

\bib{CHM}{article}{
      author={Caporaso, Lucia},
      author={Harris, Joe},
      author={Mazur, Barry},
       title={Uniformity of rational points},
        date={1997},
     journal={J. Amer. Math. Soc.},
      volume={10},
      number={1},
       pages={1\ndash 35},
}

\bib{EEHK09}{article}{
      author={Ellenberg, Jordan~S.},
      author={Elsholtz, Christian},
      author={Hall, Chris},
      author={Kowalski, Emmanuel},
       title={{Non-simple abelian varieties in a family: geometric and analytic
  approaches}},
        date={2009},
        ISSN={0024-6107},
     journal={Journal of the London Mathematical Society},
      volume={80},
      number={1},
       pages={135\ndash 154},
  eprint={https://academic.oup.com/jlms/article-pdf/80/1/135/2605817/jdp021.pdf},
         url={https://doi.org/10.1112/jlms/jdp021},
}

\bib{EV05}{article}{
      author={Ellenberg, J.},
      author={Venkatesh, A.},
       title={On uniform bounds for rational points on nonrational curves},
        date={2005},
        ISSN={1073-7928},
     journal={International Mathematics Research Notices},
      volume={2005},
      number={35},
       pages={2163\ndash 2181},
  eprint={https://academic.oup.com/imrn/article-pdf/2005/35/2163/1828271/2005-35-2163.pdf},
         url={https://doi.org/10.1155/IMRN.2005.2163},
}

\bib{Fre97}{inproceedings}{
      author={Frey, Gerhard},
       title={On ternary equations of fermat type and relations with elliptic
  curves},
        date={1997},
}

\bib{FuNonsimple}{misc}{
      author={Fu, Yu},
       title={Non-simple abelian varieties in a family: arithmetic approaches},
        date={2024},
         url={https://arxiv.org/abs/2412.11048},
}

\bib{Hall}{book}{
      author={Hall, M.},
       title={The {T}heory of {G}roups},
      series={Dover Books on Mathematics},
   publisher={Dover Publications},
        date={2018},
        ISBN={9780486816906},
         url={https://books.google.com/books?id=K8hEDwAAQBAJ},
}

\bib{Heath-Brown}{article}{
      author={Heath-Brown, D.~R.},
       title={The {D}ensity of {R}ational {P}oints on {C}urves and {S}urfaces},
        date={2002},
        ISSN={0003486X},
     journal={Annals of Mathematics},
      volume={155},
      number={2},
       pages={553\ndash 598},
         url={http://www.jstor.org/stable/3062125},
}

\bib{HK00}{article}{
      author={Halberstadt, E.~S.},
      author={Kraus, Alain},
       title={On the modular curves {YE(7)}},
        date={2000},
     journal={Math. Comput.},
      volume={69},
       pages={1193\ndash 1206},
}

\bib{Maz78}{article}{
      author={Mazur, B.},
      author={Goldfeld, D.},
       title={Rational isogenies of prime degree},
        date={1978},
     journal={Inventiones mathematicae},
      volume={44},
      number={2},
       pages={129\ndash 162},
         url={https://doi.org/10.1007/BF01390348},
}

\bib{Paz19}{article}{
      author={Pazuki, Fabien},
       title={Modular invariants and isogenies},
        date={2019},
     journal={International Journal of Number Theory},
      volume={15},
      number={03},
       pages={569\ndash 584},
}

\bib{PS22}{article}{
      author={Paredes, Marcelo},
      author={Sasyk, Rom\'an},
       title={Uniform bounds for the number of rational points on varieties
  over global fields},
        date={2022},
     journal={Algebra \& Number Theory},
      volume={16},
      number={8},
       pages={1941\ndash 2000},
}

\bib{RV00}{article}{
      author={REVERTER, AMADEU},
      author={VILA, NÚRIA},
       title={Galois representations attached to the product of two elliptic
  curves},
        date={2000},
        ISSN={00357596, 19453795},
     journal={The Rocky Mountain Journal of Mathematics},
      volume={30},
      number={3},
       pages={1121\ndash 1127},
         url={http://www.jstor.org/stable/44238527},
}

\bib{Sal23}{article}{
      author={Salberger, Per},
       title={Counting rational points on projective varieties},
        date={2023},
     journal={Proceedings of the London Mathematical Society},
      volume={126},
}

\bib{Ser}{book}{
      author={Serre, Jean-Pierre},
       title={Lectures on the mordell-weil theorem},
      series={Aspects of Mathematics},
   publisher={Vieweg+Teubner Verlag Wiesbaden},
        date={1989},
        ISBN={978-3-663-10632-6},
         url={https://doi.org/10.1007/978-3-663-10632-6},
}

\bib{Sil09}{book}{
      author={Silverman, J.H.},
       title={The {A}rithmetic of {E}lliptic {C}urves},
      series={Graduate Texts in Mathematics},
   publisher={Springer New York},
        date={2009},
        ISBN={9780387094946},
         url={https://books.google.com/books?id=Z90CA\_EUCCkC},
}

\bib{Walsh}{article}{
      author={Walsh, Miguel~N.},
       title={Bounded {R}ational {P}oints on {C}urves},
        date={2014},
        ISSN={1073-7928},
     journal={International Mathematics Research Notices},
      volume={2015},
      number={14},
       pages={5644\ndash 5658},
  eprint={https://academic.oup.com/imrn/article-pdf/2015/14/5644/2344558/rnu103.pdf},
         url={https://doi.org/10.1093/imrn/rnu103},
}

\end{biblist}
\end{bibdiv}

\end{document}